\documentclass[11pt,a4paper]{article}
\usepackage[latin1]{inputenc}   

\usepackage{amsfonts,amsmath,layout}
\usepackage{mathrsfs}  
\usepackage{amssymb,mathabx,amsfonts,amsthm,amscd,stmaryrd,dsfont,esint,upgreek,constants,todonotes}

\usepackage{makeidx}
\usepackage[T1]{fontenc}
\usepackage{amsmath}
\usepackage{amsthm}
\usepackage{amsfonts}
\usepackage{amssymb}
\usepackage{alltt}
\usepackage{color}
\usepackage{graphicx}
\usepackage{ulem}

\newcommand{\N}{\mathbb N}

\newcommand{\Z}{\mathbb Z}

\newcommand{\R}{\mathbb R}

\def\P{\mathbb P}
\def\F{\mathbb F}

\def\Pw{{\mathcal P}(\T^d)}
\def\Pk{{\mathcal P}(\T^d)}
\def\dk{{\bf d}_1}

\newcommand{\be}{\begin{equation}}
\newcommand{\ee}{\end{equation}}
\def\1{{\bf 1}}

\def\ep{\epsilon}

\def\Dt0{{\bf D}(t_0)}

\def\E{{\bf E}}

\def\P{{\bf P}}

\def\to{\rightarrow}

\def\ds{\displaystyle}

\def\ds{\displaystyle}

\def \to{\rightarrow}

\def \T{\mathbb{T}}
\def \R {\mathbb{R}}
\def \N {\mathbb{N}}
\def \Z {\mathbb{Z}}

\def\dk{{\bf d}_1}
\def\dive{{\rm div}}

\def \ep{\varepsilon}

\def\E{\mathbb E}
\def\P{\mathbb P}
\def\inte{\int_{\T^d}}
\definecolor{ProcessBlue}{cmyk}{1,0,0,0.40}

\newtheorem{Theorem}{Theorem}[section]
\newtheorem{Definition}[Theorem]{Definition}
\newtheorem{Proposition}[Theorem]{Proposition}
\newtheorem{Lemma}[Theorem]{Lemma}

\title{An example of multiple mean field limits in ergodic differential games}
\author{Pierre Cardaliaguet\thanks{Universit\'{e} Paris-Dauphine, PSL Research University, Ceremade. cardaliaguet@ceremade.dauphine.fr} \and Catherine Rainer\thanks{Universit\'{e} de Bretagne Occidentale, LMBA. Catherine.Rainer@univ-brest.fr
}}

\begin{document}

\maketitle

\begin{abstract} We present an example of symmetric ergodic $N-$players differential games, played in memory strategies  on the position of the players, for which the limit set, as $N\to +\infty$, of Nash equilibrium payoffs is large, although the game has a single mean field game equilibrium. This example is in sharp contrast with a result by Lacker \cite{La18} for finite horizon problems. 
\end{abstract}


\section*{Introduction}

In this note we want to underline the role of information in mean field games. For this we study the limit of Nash equilibrium payoffs in ergodic $N-$player stochastic differential games as the number $N$ of players tends to infinity. Since the pioneering works by Lasry and Lions \cite{LL07mf} (see also \cite{HMC}) differential games with many agents have attracted a lot of attention under the terminology of mean field games. We also refer the reader to the monographs \cite{bensoussan2013mean, CDbook}. Mean field games are nonatomic  dynamic games, in which the agents interact through the population density.  

Here we investigate in what extend the mean field game problem is the limit of the $N-$person differential games. This question is surprisingly difficult in general and is not completely understood so far in full generality. When, in the $N-$player game, players play in open-loop (i.e., observe only their own position but not the position of the other players), the mean field limit is a mean field game. The first result in that direction goes back to \cite{LL07mf} in the ergodic setting (see also \cite{ABC, Fe} for statements in the same direction); extensions to the non Markovian setting can be found in Fischer \cite{Fi} while Lacker gave a complete characterization of the limit \cite{La16} (see also \cite{Nu} for an exit time problem). Note that these (often technically difficult) results are not entirely surprising since, in the N-player game as well as in the mean field game, the players do not observe the position of the other players:  therefore there is no real change of nature between the $N-$player problem and the mean field game. 

We are interested here in the $N-$player games in which players   observe each other,  the so-called closed-loop regime. In this setting, the mean field limit is much less understood and one possesses only partial results. In general, one formalizes the closed-loop Nash equilibria in the $N-$person game by a PDE (the Nash system) which describes the fact that players react in function of the current position of all the other players. The first convergence result in this setting \cite{CDLL} states that, in the finite horizon problem and under a suitable monotonicity assumption, the solution of the Nash system converges to a MFG equilibrium. The convergence relies on the construction of a solution to the so-called master equation, a partial differential equation stated in the space of probability measures. The result was later applied and extended to different frameworks, with similar---or closely related---techniques of proof in \cite{BaCo, Ca17, CePe, CDPFP, DLR, DLR2}. 

Recently Lacker \cite{La18} introduced completely different arguments to handle the problem. He proved the convergence of the closed-loop Nash equilibria to ``extended'' MFG equilibria. Even more surprisingly, his result extends to generalized Markov strategies, where players can remember the past positions of the other players (memory strategies). The key remark is that, for a large number of players and ``in average'', the fact that a single player deviates does not change too much the (time dependent) distribution of the players. Note that \cite{La18} holds in a set-up in which the noise of each player is non degenerate. 

It is important to point out that the result of \cite{La18} cannot be extended to strategies in which players observe the controls of the other players. Indeed, the so-called Folk Theorem \cite{BCR, Ko76} states that players can detect and ``punish''  a player who deviates and therefore, even when the number of players is large, the behavior of a single player completely changes the outcome of the game. A way to understand \cite{La18} is  that, because of the (nondegenerate) noise, the observation of a player's position does not give information on the fact that this player has deviated or not.

The aim of the present paper is to address a similar question for (a particular class of) ergodic differential games. Let us first recall that, in the open-loop regime, limits of Nash equilibria in the $N-$player game are MFG equilibria \cite{ABC, Fe, LL07mf}. On the other hand, in the closed loop Markovian regime, the convergence problem is surprisingly open up to now, although the existence of a solution to the ergodic master equation is known \cite{CaPo}: Indeed in this ergodic set-up, the use of the solution to the master equation is not obvious and the technique of proof of \cite{CDLL} does not seem to apply. Here we concentrate on the limit of equilibria in $N-$player differential games with generalized Markov strategies. While,  for the finite horizon problems, these Nash equilibria always converge to MFG equilibria \cite{La18}, we show  that this is no longer the case in the ergodic regime. 

This means that, when the horizon becomes infinite, players can learn from the other players even if they observe their positions only. Our convergence result is reminiscent of the Folk Theorem of \cite{BCR}, but in a framework of an ergodic cost and in which players observe only the positions of the other players. 

In order to explain more precisely our result, let us describe the framework in which we work. We consider $N-$player differential games played in strategies depending on the past positions of all the players (See Subsection \ref{subsec.Nash} below). Player $i$ (where $i\in \{1, \dots, N\}$)  minimizes an ergodic and symmetric cost of the form 
$$
\limsup_{T\to+\infty} \frac{1}{T} \E\left[ \int_0^T (L(\alpha^i_t, X^i_t)+ F( m^{N,i}_{{\bf X}_t}))dt\right]
$$
where, for any $j\in \{1, \dots, N\}$, $X^j$ is the position of player $j$ at time $t$, $\alpha^j$ is the control of player $j$,  ${\bf X}=(X^1, \dots, X^N)$, 
$$
m^{N,i}_{{\bf X}_t}:= \frac{1}{N-1}\sum_{j\neq i} \delta_{X^j_t}
$$
is the empirical measure of all players but player $i$ and the dynamic of $X^j$ is just 
$$
dX^j_t=\alpha^j_t dt +dB^j_t,
$$
where the $B^j$ are independent $d-$dimensional Brownian motion. Moreover, $F:\Pk\to \R$ is a sufficiently smooth map (where $\Pw$ is the set of Borel probability measures on $\T^d$). Note that we work here with periodic data (and thus in the $d-$dimensional torus $\T^d$). 

In this setting the mean field game payoff is unique and given by $e_{MFG}:=- \lambda_0+ F( \mu_0)$, where $\mu_0$ is the unique invariant measure solution to the equation 
$$
-\frac12 \Delta \mu_0 -\dive (\mu_0 H_p(Du_0(x),x))=0 
$$
and the pair $(u_0, \lambda_0)$ is the unique solution to the ergodic Hamilton-Jacobi equation: 
$$
-\frac12 \Delta u_0+ H(Du_0,x)= \lambda_0 \qquad {\rm in }\; \T^d
$$
where 
$$
H(p,x)=\sup_{a\in \R^d} -a\cdot p -L(a,x)
$$
(see \cite{CaPo}).\\
On the other hand, the ``social cost'' (i.e., the smallest cost a global planner can achieve, see \cite{CaRa}) is given by 
$$
e_{\min} :=  \inf_{\mu,\alpha} \int_{\T^d} L(\alpha(x),x)\mu(dx) +F(\mu), 
$$
where the infimum is taken over all pair $(\mu,\alpha)$ where $\mu$ is the invariant measure on $\T^d$ associated with the distributed control $\alpha:\T^d\to \R^d$, i.e., satisfying the equation 
$$
-\frac12 \Delta \mu+ \dive(\mu \alpha)= 0 \; {\rm in }\; \R^d. 
$$
Our result states that, for any 
\be\label{eq.minaresultintro}
e\in [ e_{\min}, -\lambda_0 +\max F),
\ee
there exists a symmetric Nash equilibrium payoff $(e^N, \dots, e^N)$ in the $N-$player game such that $e^N$ converges to $e$. Unless $F$ is constant, the interval in \eqref{eq.minaresultintro} has a non empty interior:
$$
e_{min} < e_{MFG} \leq -\lambda_0 +\max F
$$ 
(see \cite{CaRa}). So the limit of the $N-$player game contains many more Nash equilibrium payoffs than the MFG one, including the social cost. \\

Let us underline again that our result says nothing on the convergence, as $N\to +\infty$, of the solution $(v^{N,i})$ of the $N-$player Nash system 
$$
- \frac{1}{2}\sum_{j=1}^N \Delta_{x_j} v^{N,i}({\bf x}) +H(D_{x_i}v^{N,i}({\bf x}), x_i) +\sum_{j\neq i} D_{x_j}v^{N,i}({\bf x}) \cdot H_p(D_{x_j}v^{N,j}({\bf x}),x_j)= F(m^{N,i}_{\bf x})+\lambda^{N,i} 
$$
for $i=1, \dots, N$ and ${\bf x}=(x_1, \dots, x_N)\in (\T^d)^N$. This convergence is, so far, an open problem. \\

The paper is organized as follows: in the first section, we state our assumptions and introduce the main definitions (Nash equilibria in generalized Markov strategies, mean field game equilibria, social cost). The second section is dedicated to the statement and the proof of the existence of several mean field limits. \\

{\bf Acknowledgement:} The authors were partially supported by the ANR (Agence Nationale de la Recherche) project ANR-16-CE40-0015-01. The first author was partially supported by the Office for Naval Research Grant N00014-17-1-2095.

\section{Notation, assumption and basic definitions} 

\subsection{Notation and assumptions}

Our game takes place in $\R^d$. However our data are $\Z^d-$periodic  in space, which means that we mostly work in the $d-$dimensional torus $\T^d=\R^d/\Z^d$ and we denote by $\pi:\R^d\to \T^d$ the natural projection. Given a Borel probability measure $m$ on $\R^d$, we often project it into the set ${\mathcal P}(\T^d)$ of Borel probability measures on $\T^d$ by considering $\pi\sharp m$ defined by 
$$
\int_{\T^d} f(x) \pi\sharp m(dx):= \int_{\R^d} f(\pi(x))m(dx) \qquad \forall f\in C^0(\T^d). 
$$

Our problem  involves the following data: The Lagrangian $L:\R^d\times \T^d\to \R$ satisfies, for some constant $C_0>0$:
\be\label{HypL}
\mbox{\rm $L=L(\alpha, x)$ is of class $C^2$, with $C_0^{-1} I_d \leq D^2_{\alpha,\alpha} L(\alpha,x)\leq C_0 I_d $.} 
\ee
It will often be convenient to extend $L$ to $\R^d\times \R^d$ by setting $L(\alpha,x)= L(\alpha, \pi(x))$. The map $F: \Pk\to \R$ satisfies 
\be\label{HypF}
\mbox{\rm $F$ is of class $C^1$ with $y\to \frac{\delta F}{\delta m}(m,y)$ bounded in $C^2(\T^d)$ uniformly in $m$.}
\ee
Let us recall that for $F$ to be of class $C^1$ means that there exists a continuous map $\frac{\delta F}{\delta m}:\Pk\times \T^d \to \R$ such that 
$$
F(m')-F(m)= \int_0^1 \inte \frac{\delta F}{\delta m}((1-s)m+sm',y)(m'-m)(dy)ds \qquad \forall m,m'\in \Pk. 
$$
In particular, $F$ is continuous in $m$. It will be convenient to assume that 
\be \label{HypF2}
\mbox{\rm $F$ is not constant on $\Pk$.}
\ee
Throughout the paper and for any $N\in \N$, the initial condition $(x^1_0, \dots, x^N_0)\in (\T^d)^N$ is fixed (and actually irrelevant). 

\subsection{Nash equilibria in generalized Markov strategies} \label{subsec.Nash}

In the  $N-$player game, players play  nonanticipative strategies on the trajectory of the other players: namely, a strategy of player $i\in \{1, \dots, N\}$ is a bounded map $\alpha^i: \R_+ \times (C^0(\R_+, \R^d))^{N}\to \R^d$ which is Borel measurable and such that, for any $(x_j)_{j=1, \dots, N}$ and $(y_j)_{j=1, \dots, N}$ in $(C^0(\R_+, \R^d))^{N}$ which coincide on $[0, t]$, we have 
$$
\alpha^i((x_j)_{j=1, \dots, N})_s= \alpha^i((y_j)_{j=1, \dots, N})_s\qquad \mbox{\rm a.e. on } [0, t].
$$
These strategies are called ``generalized Markov strategies''. We denote by ${\mathcal A}$ the set of generalized Markov strategies for a player (note that it does not depend on $i$ since all the players are symmetric). \\
For given $(\alpha^1, \dots, \alpha^N ) \in {\mathcal A}^N$ and an initial condition ${\bf x_0}= (x^1_0,\dots x^N_0)\in (\R^d)^N$, let us consider the SDE
\be\label{eq.motion}
\left\{\begin{array}{l}
\ds dX^i_s= \alpha^i(X^1_\cdot, \dots, X^N_\cdot)ds + dB^i_s , \qquad s\geq 0,\; i\in \{1, \dots, N\},\\ 
\ds X^i_0= x^i_0 \qquad i\in \{1, \dots, N\}.
\end{array}\right.
\ee
Using Girsanov's theorem, we can find a filtered probability space $(\Omega,{\cal F},\F=({\cal F}_t)_{t\geq 0},\P)$ and, on this space, a couple of processes $({\bf X}_t,B_t)_{t\geq 0}$ with values in $\R^{d\times N}\times\R^{d\times N}$, such that ${\bf X}=(X^1, \dots, X^N)$ is adapted to $\F$, $B$ is a $\F$-Brownian motion and equation \eqref{eq.motion} is satisfied for all $i$. Moreover the couple $({\bf X},B)$ is unique in law. In what follows, as in Lacker \cite{La18}, we keep in mind that for each $N$-uple of strategies, the process ${\bf X}$ is defined on a different probability space, but, for the comfort of the reader, we don't mention this dependence in the notations.\\
In this game, we assume that the payoff of players takes the form of an  ergodic cost of mean field type. Given a family of strategies ${\bf \alpha}=(\alpha^1, \dots, \alpha^N ) \in {\mathcal A}^N$, the cost  $J^i$ of player  $i\in \{1, \dots, N\}$ is 
$$
J^i({\bf \alpha})= \limsup_{T\to+\infty} \frac{1}{T} \E\left[ \int_0^T (L(\alpha^i_t, X^i_t) + F(\pi\sharp m^{N,i}_{{\bf X}_t}) ) dt\right]
$$
where ${\bf X}=(X^1, \dots, X^N)$ is the solution to \eqref{eq.motion}, $m^{N,i}_{{\bf X}_t}$ is the empirical measure of all players but player $i$:
\be\label{defem}
m^{N,i}_{{\bf X}_t}:= \frac{1}{N-1}\sum_{j\neq i} \delta_{X^j_t}.
\ee

In this setting, the definition of a symmetric Nash equilibrium payoff is the following: 
\begin{Definition}\label{def.Nash} Fix a symmetric initial position ${\bf x_0}= (x_0,\dots x_0)\in (\R^d)^N$. We say that $e\in \R$ is a symmetric Nash equilibrium payoff (in generalized Markov strategies) if, for any $\ep>0$, there exists $\bar \alpha^1\in {\mathcal A}$ a strategy for player $1$, which is symmetric with respect to the other players: 
$$
\bar \alpha^1 (x^1, x^2,  \dots, x^N)= \bar \alpha^1 (x^1, x^{\sigma(2)},  \dots, x^{\sigma(N)})\qquad \forall x^1, \dots, x^N\in C^0(\R_+, \R^d), 
$$
for any permutation $\sigma$ on $\{2, \dots, N\}$ and such that, if we define the strategy of player $i$ by 
$$
\bar \alpha^i(x^1, \dots, x^N):= \bar \alpha^1(x^i, x^2, \dots, x^{i-1}, x^1, x^{i+1}, \dots, x^N), 
$$
then  ${\bf \bar \alpha}:= (\bar \alpha^1, \dots, \bar \alpha^N)$ is an $\ep-$Nash equilibrium with payoff close to $e$: for any $i\in \{1, \dots, N\}$, 
$$
J^i({\bf \bar \alpha}) \leq J^i(\alpha^i, (\bar \alpha^j)_{j\neq i} )+\ep \qquad \forall \alpha^i\in {\mathcal A}
$$
and 
$$
\left| J^i({\bf \bar \alpha})-e\right| \leq \ep. 
$$
\end{Definition}

\subsection{The ergodic MFG equilibrium}

As the number $N$ of players tends to infinity, one often expects that the limit of an $N-$player differential game becomes a mean field game. As our game is of ergodic type, the MFG equilibrium takes  the form of the following ergodic MFG, in which the unknown are $(\lambda, u,\mu)$:
\be\label{eq.MFGergo}
\left\{\begin{array}{l}
\ds -\frac12 \Delta u+ H(Du,x)= F(\mu)+\lambda \qquad {\rm in }\; \T^d,\\
\\
\ds -\frac12 \Delta \mu-{\rm div} (\mu H_p(Du,x))=0\qquad {\rm in }\; \T^d, \\
\ds \mu\geq 0, \; \inte\mu=1. 
\end{array}\right.
\ee

It turns out that, for our problem, the unique MFG equilibrium has a very simple structure:

\begin{Proposition} The unique MFG equilibrium is given by $(\lambda_0-F(\mu_0),u_0,\mu_0)$ where  $(\lambda_0,u_0)$ solve the ergodic problem 
$$
-\frac12 \Delta u_0+ H(Du_0,x)= \lambda_0 \qquad {\rm in }\; \T^d
$$
($u_0$ is unique up to constants) and $\mu_0$ is the unique probability measure such that 
$$
-\frac12 \Delta \mu_0-{\rm div} (\mu_0 H_p(Du_0,x))=0 \qquad {\rm in }\; \T^d. 
$$
\end{Proposition}

In this case, the payoff of the MFG equilibrium is given by 
$$
e_{MFG}:= -\lambda_0+F(\mu_0).
$$ 

\begin{proof} The proof is immediate, since,  $(\lambda, u, \mu):= (\lambda_0-F(\mu_0),u_0,\mu_0)$ satisfies the MFG system \eqref{eq.MFGergo} and the solution of this system is unique because $F$, being independent of $x$,  satisfies the standard Lasry-Lions monotonicity condition (see \cite{LL07mf}). 
\end{proof}

Let us recall for later use that the measure $\nu_0:= (Id,-H_p(Du_0))\sharp \mu_0$ minimizes the energy 
\be\label{eq.characlambda0}
-\lambda_0 = \int_{\R^d\times \T^d}  L(a,x) \nu_0(da,dx)= \inf_\nu \int_{\R^d\times \T^d} L(a,x) \nu(da,dx) 
\ee
where the infimum in the last term is computed among all the Borel probability measures on $\R^d\times \T^d$ which are closed (see \cite{GV}): 
\be\label{closedcond}
\int_{\R^d\times \T^d} (a\cdot D\phi + \frac12\Delta \phi) \nu(da,dx) = 0\qquad  \forall \phi\in {\mathcal C}^\infty(\T^d). 
\ee

\subsection{The social cost} \label{subsec.soccost}

A last notion of interest in our problem is the social cost, or the cost for a global planner. It takes the form 
\be\label{defemin}
e_{\min}:= \inf_{(m,\alpha)} \int_{\R^d} L(\alpha(x),x) m(dx) + F(m),
\ee
where the infimum is computed among the pairs $(m,\alpha)$, where $m\in \Pk$ and $\alpha\in L^2_m(\T^d, \R^d)$ satisfy the constraint
\be\label{continuityeq}
-\frac{1}{2}\Delta m +\dive(m\alpha)= 0 \qquad {\rm in} \; \T^d. 
\ee
We will often use the fact that, if $(m,\alpha)$ is as above, then the measure $\nu:= (Id, \alpha)\sharp m$ satisfies \eqref{closedcond} and therefore 
$$
\inte L(\alpha(x),x)m(dx) \geq - \lambda_0. 
$$
It is known (see \cite{GV}) that, under our continuity and growth assumption on $L$ and $F$ in \eqref{HypL} and \eqref{HypF}, the problem has at least one solution $(\tilde m, \tilde \alpha)$ and that, under the differentiability assumption \eqref{HypF} on $F$, there exists $\tilde u:\T^d\to \R$ of class $C^2$ such that the pair $(\tilde u, \tilde m)$ satisfies the (new) MFG system 
$$
\left\{\begin{array}{l} 
-\frac{1}{2}\Delta \tilde u +H(D\tilde u,x)=\frac{\delta F}{\delta m}(\tilde m,x) \qquad {\rm in }\; \T^d\\
\\
-\frac{1}{2}\Delta \tilde m -\dive (\tilde m H_p(D\tilde u,x))=0 \qquad {\rm in }\; \T^d\
\end{array}\right.
$$
with $\tilde \alpha(x)= -H_p(D\tilde u(x),x)$. In view of the regularity of $\tilde u$, $\tilde m$ is at least of class $C^2$ and is positive.

In a previous paper \cite{CaRa} (in a more general set-up than here, for time dependent MFGs), we have proved that there is no equality between the MFG and the social cost, unless $F$ is constant: More precisely, if $F$ satisfies \eqref{HypF2}, then 
\be \label{condcond}
e_{\min}< e_{MFG} \leq e_{\max}:= - \lambda_0+\max_{m\in \Pk} F(m). 
\ee

\section{The convergence result}

We explain here that, for our problem, one cannot expect the convergence of all the Nash equilibria in generalized Markov strategies to the MFG equilibrium:  

\subsection{The main theorem} 

\begin{Theorem}\label{thm:main} Under the assumptions \eqref{HypL}, \eqref{HypF} and \eqref{HypF2} on $L$ and $F$, and for any 
$$
e\in [e_{\min}, e_{\max}), 
$$
there exists a sequence of symmetric Nash equilibrium payoffs $(e^N)$ in the $N-$player game such that 
\be\label{liminfeN}
\lim_{N\to +\infty} e^N =e. 
\ee
\end{Theorem}

Let us recall that $e_{\min}$, defined by \eqref{defemin}, corresponds to the ``social cost'', while $e_{\max}$, introduced in \eqref{condcond}, is not smaller than the cost associated with the mean field game. In view of \eqref{condcond}, Theorem \ref{thm:main} implies that the limit of symmetric Nash equilibria in generalized Markov strategies is not necessarily an MFG equilibrium. This is in sharp contrast with the finite horizon problem studied by Lacker in \cite{La18}. 

The construction of the symmetric Nash equilibrium payoffs $e^N$ is based on the ``Folk Theorem'' in differential games: see \cite{BCR}. In general, the Folk Theorem is related to the observation of the control of the players. Surprisingly here, only the observation of the position of the other players is necessary: this is specific to the ergodic cost (and of the particular structure of our game). 

\subsection{Proof of the main theorem}

As $F$ and $L$ are bounded below, we can assume without loss of generality that 
\be\label{FLgeq0}
L\geq 0, \; F\geq 0.
\ee

The first step consists in showing that the cost $e$ can be achieved by a suitable stationary solution $(\hat m, \hat \alpha)$: 

\begin{Lemma}
	\label{lem.zlejgsf}
 There exists $(\hat m, \hat \alpha)$, of class $C^1$, with $\hat m$ a probability measure on $\T^d$ with a positive density, satisfying \eqref{continuityeq} and such that 
\be\label{repe}
e  = \min_\alpha \inte L( \alpha(x),x)\hat m(x)dx +F(\hat m)=  \inte L(\hat \alpha(x),x)\hat m(x)dx +F(\hat m),
\ee
where the infimum is taken over the vector fields $\alpha\in L^2_{\hat m}(\T^d,\R^d)$ such that $(\hat m, \alpha)$ satisfies \eqref{continuityeq}. 

In addition, there exists a sequence $(m^n,\alpha^n)$, of class $C^1$,  with $m^n$ a probability measure on $\T^d$ with a positive density, satisfying \eqref{continuityeq} and such that 
$$
\lim_n F(m^n)=\max F, \qquad  \min_\alpha \inte L( \alpha(x),x) m^n(x)dx =  \inte L( \alpha^n(x),x) m^n(x)dx. 
$$
\end{Lemma}

\begin{proof} In a first step, we show that there exists $(\hat m', \hat \alpha')$ of class $C^1$, with $\hat m'$ a probability measure on $\T^d$ with a positive density, satisfying \eqref{continuityeq} and such that
\be\label{khjelnkrsfdc}
e  < \min_\alpha \inte L( \alpha(x),x)\hat m(x)dx +F(\hat m')=  \inte L(\hat \alpha'(x),x)\hat m'(x)dx +F(\hat m'),
\ee
where the infimum is taken over the vector fields $\alpha\in L^2_{\hat m'}(\T^d,\R^d)$ such that $-\Delta \hat m'+{\rm div}(\hat m' \alpha)=0$ in $\T^d$. For proving this, we now build the sequence $(m^n,\alpha^n)$, where, for each $n$, $(m^n,\alpha^n)$ is a minimum of the problem:
\begin{align*}
\inf\Bigl\{ n^{-1}\inte L(\alpha(x),x)m(x)dx -F(m), \;{\rm where }\;  m\in \Pw, \; \alpha \in L^2_{m}(\T^d, \R^d), \qquad & \\
 -\Delta m+{\rm div} (m\alpha)=0\; {\rm in }\; \T^d\Bigr\}. & 
\end{align*}
Let us recall that such a minimum exists (see Subsection \ref{subsec.soccost}). In addition,  there exists $u^n$ such that $(u^n, m^n)$ solves the MFG system 
$$
\left\{\begin{array}{l}
- \frac12\Delta u^n +H(Du^n,x)= - n\frac{\delta F}{\delta m}(m^n,x),\\
\\
- \frac12\Delta m^n -{\rm div}( m^n H_p(Du^n, x))= 0,
\end{array}\right.
$$
with $\alpha^n= -H_p(Du^n, x)$. In particular, $(m^n, \alpha^n)$ is of class $C^1$ and $m^n$ has a positive density. Next we claim that 
$$
\lim_n F(m^n)=  \max F.
$$
Indeed, we can find another sequence $(\mu^k)$ of smooth and positive probability densities on $\T^d$ such that $(F(\mu^k))$ converges to $\max F$.  Let us set $\beta^k:=D(\log(\mu^k))$. Then $(\mu^k, \beta^k)$ satisfies the constraint \eqref{continuityeq} and therefore, for any $k$ and by the optimality of $(m^n,\alpha^n)$, 
\begin{align*}
\max F\geq \limsup_n F(m^n) & \geq \liminf_n F(m^n) \geq  \liminf_n \Bigl( - n^{-1}\inte L(\alpha^n(x),x)m^n(x)dx + F(m^n)   \Bigr) \\
& \geq  \liminf_n  \Bigl( - n^{-1}\inte L(\beta^k(x),x)\mu^k(x)dx + F(\mu^k)   \Bigr)  = F(\mu^k).
\end{align*}
Letting $k\to+\infty$ proves that $(F(m^n))$ converges to $\max F$. Then, recalling the  characterization of $\lambda_0$ in \eqref{eq.characlambda0}, we have
$$
e< -\lambda_0+\max F \leq   \liminf_n \inte L(\alpha^n(x),x)m^n(x)dx + F(m^n).
$$
So, for $n$ large enough, we have 
$$
e<  \inte L(\alpha^n(x),x)m^n(x)dx + F(m^n).  
$$
Setting $(\hat m', \hat \alpha'):= (m^n,\alpha^n)$ for such a large $n$ proves the first step. We also set for later use $\hat u':= u^n$ and recall that $\alpha^n= -H_p(D\hat u',x)$. 

We now build the pair $(\hat m, \hat \alpha)$ required in the lemma. For $\lambda \in [0,1]$, let $m^\lambda$ be the unique invariant measure associated with the vector field  $\alpha^\lambda (x):= -H_p(D\phi^\lambda(x),x)$, where $\phi^\lambda :=  (1-\lambda) \tilde u +\lambda \hat u'$. 
Note that $m^\lambda$ is unique, of class $C^1$ and has a positive density since $\alpha^\lambda$ is of class $C^1$. Moreover, $\lambda \to m^\lambda$ is continuous in $C^1$ by the same argument. Next we note that $\alpha^\lambda$ is a minimum of the problem: 
$$
\inf\left\{ \inte L(\alpha(x),x)m^\lambda(dx), \;{\rm where }\;  \alpha \in L^2_{m^\lambda}(\T^d, \R^d),\; -\Delta m^\lambda+{\rm div} (m^\lambda\alpha)=0\; {\rm in }\; \T^d\right\} .
$$
Indeed it is well-known that a vector field $\alpha$ is a minimum of this problem  if and only if there exists $\phi\in H^1(\T^d)$ such that $\alpha= -H_p(D\phi,x)$ and the pair $(m^\lambda, -H_p(D\phi,x))$ satisfies \eqref{continuityeq}: this is indeed the case for $\alpha^\lambda$ by construction. 

As the map $\lambda \to \inte L(\alpha^\lambda,x)m^\lambda+F(m^\lambda)$ is continuous and as it is equal to $e_{\min}$ (which is not larger than $e$) for $\lambda=0$ and---by \eqref{khjelnkrsfdc}---is not smaller than $e$ for $\lambda=1$, we can find $\lambda\in [0,1]$ such that $e=\inte L(\alpha^\lambda,x)m^\lambda+F(m^\lambda)$. We set $(\hat m, \hat \alpha):= (m^\lambda,\alpha^\lambda)$ from now on and $(\hat m, \hat \alpha)$ satisfies the required conditions. 
\end{proof}

 A second lemma shows that $\hat\alpha$ remains $\epsilon$-optimal in \eqref{repe}, after a sufficiently small perturbation of $\hat m$.

\begin{Lemma}
	\label{lemma2}
	For any $\ep>0$ there exists $\delta>0$ (depending on $\ep$), such that, for any closed Borel probability measure $\sigma$ on $\R^d\times \T^d$ with second marginal $\mu$ and with ${\bf d}_1(\mu, \hat m
		)\leq \delta$, one has 
		$$
		\int_{\R^d\times \T^d} L( a, x) \sigma (da, dx) \geq \int_{\T^d} L(\hat \alpha(x), x) \hat m(dx) -\ep. 
		$$
	\end{Lemma}

\begin{proof}
  By the definition of $(\hat m,\hat \alpha)$, $\hat \alpha$ minimizes the quantity
$$
\inte L(x,\alpha(x)) \hat m(x),
$$
where the infimum is taken over the  maps $\alpha\in L^2_{\hat m}(\T^d)$ such that 
$$
-\Delta \hat m +{\rm div} (\hat m \alpha)=0. 
$$
Let us now argue by contradiction and assume that there exists a sequence $(\sigma^n)$ of closed measures, with second marginal $m^n$ converging to $\hat m$ and with
\be\label{lqzeflkjn}
\int_{\R^d\times \T^d} L( a, x) \sigma^n (da, dx) < \int_{\T^d} L(\hat \alpha(x), x) \hat m(dx) -\ep. 
\ee
In view of the coercivity of $L$ with respect to the first variable, the sequence $\sigma^n$ is tight and there exists a subsequence (still labelled in the same way) which converges to some measure $\sigma$ on $\R^d\times \T^d$. Note that $\sigma$ is closed (as the limit of the $(\sigma^n)$ which are closed) and its second marginal is $\hat m$. Let us disintegrate $\sigma$ with respect to $\hat m$: $\sigma=  \sigma_x(d\alpha)\hat m(dx)$ and let us set $\tilde \alpha(x):= \inte \alpha \sigma_x(d\alpha)$. Then $(\hat m, \tilde \alpha)$  satisfies \eqref{continuityeq} since the measure $\sigma$ is closed. In addition, by convexity of $L$ with respect to the first variable and \eqref{lqzeflkjn}, 
\begin{align*}
\inte L(\tilde\alpha(x),x) \hat m(dx) &  \leq \int_{\R^d\times \T^d} L( a, x) \sigma_x (da) \hat m(dx) \leq  \limsup_n  \int_{\R^d\times \T^d} L( a, x) \sigma^n (da, dx) \\
&\qquad  \leq \int_{\T^d} L(\hat \alpha(x), x) \hat \mu(dx) -\ep,
\end{align*}
which contradicts the optimality of $\hat \alpha$.
\end{proof}

We now build the Nash equilibrium payoff and the corresponding strategies. Let $(\hat m, \hat \alpha)$ and $(m^n, \alpha^n)$ be as in Lemma \ref{lem.zlejgsf}. 
Let us set 
$$
e^N:= \int_{\T^d} L(\hat \alpha(x),x)\hat m(dx) + \int_{(\T^d)^{N-1}} F(m^{N,1}_{\bf x})\hat m(dx_2)\dots \hat m(dx_N).
$$
(Recall that the empirical measure $m^{N,1}_{\bf x}$ is defined by \eqref{defem}). By the Glivenko-Cantelli law of large numbers and \eqref{repe}, 
\be \label{eNcve}
\mbox{\rm $(e^N)$ converges to $e$.}
\ee
As $e< -\lambda_0+ \max F$ and by definition of $(m^n, \alpha^n)$, we can choose $n$ large enough (and fixed from now on) such that, for any $N$ large enough (given again by the Glivenko-Cantelli law of large numbers), 
\be\label{cond.ineq}
e^N \leq -\lambda_0+ \int_{(\T^d)^{N-1}} F(m^{N,1}_{\bf x})   m^n(dx_2)\dots   m^n(dx_N).
\ee

Our aim is to prove that, under the above conditions, $e^N$ is a Nash equilibrium payoff of the $N-$player game played in generalized Markov strategies. For this, we fix $\ep>0$. Given $T, \delta>0$ to be chosen below depending on $\ep$, we define  the strategies $\beta^{N,T,\delta,i}$ as follows: Given $(X^1, \dots, X^N)\in ({\mathcal C}^0(\R_+,\R^d))^N$, we define 
$$
\theta(X^1, \dots, X^N) = \inf \left\{t\geq T, \; \sup_{j\in\{1, \dots, N\} }  {\bf d}_1( \pi \sharp \left(\frac{1}{t}\int_0^t \delta_{X^j_s}ds\right), \hat m) \geq \delta\right\}
$$
(with the usual convention $\theta(X^1, \dots, X^N)=+\infty$ if the right-hand side is empty). Then we set 
$$
\beta^{N,T,\delta,i}(X^1, \dots, X^N)_t= \left\{\begin{array}{ll}
\hat \alpha(X^i_t) & {\rm if }\; t \leq \theta(X^1, \dots, X^N)\\ 
 \alpha^n(X^i_t) & {\rm otherwise}. 
\end{array}\right.
$$

We are going to show that, if $T$ is large enough and $\delta$ is small enough (depending on $\ep$), then $(\beta^{N,T,\delta,i})$ is an $\ep-$Nash equilibrium with payoff given by $e^N$.\\

\begin{Lemma}\label{lem1} The payoff of the strategies $(\beta^{N,T,\delta,i})$  is almost $e^N$: 
\be\label{payoffbetaNTdelta}
|J^i((\beta^{N,T,\delta,j})) - e^N| \leq \ep\qquad \forall i\in \{1, \dots, N\}. 
\ee
\end{Lemma}

\begin{proof} Let  $(X^i_t)$  and $(\hat X^i_t)$ be respectively the solutions to the systems
\be\label{pashatX}
d X^i_t= \beta^{N,T,\delta,i} ({\bf X}_\cdot)_t dt +dB^i_t, \; X^i_0=x^i_0, \qquad i=1, \dots, N
\ee
and 
\be\label{hatX}
d\hat X^i_t= \hat \alpha(\hat X^i_t)dt +dB^i_t, \; X^i_0=x^i_0, \qquad i=1, \dots, N,
\ee
 We can find a filtered probability space endowed with an $\R^{N\times d}$-valued Brownian motion on which, for all $i$, both \eqref{pashatX} and \eqref{hatX} admit strong solutions $X^i$ and $\hat X^i$. In particular, setting $\theta=\theta(X^1,\ldots,X^N)$, they satisfy $X^i_s=\hat X^i_s$  for all $s\geq 0$ and $i$, a.s. on the event $\{ s\leq \theta\}$.\\
Define the random time
	\[\tau=\sup\left\{ t> 0,\sup_{i=1,\dots, N} \dk (\pi\sharp \left(\frac{1}{t} \int_0^t \delta_{\hat X^i_s} ds\right) , \hat m) \geq \delta \right\},\]
	with $\sup\emptyset=0$.
Since, by the ergodic Theorem, for $\P-$a.e. $\omega\in \Omega$,
\be\label{limhatXi}
\lim_{t\to +\infty}\pi\sharp\left( \frac{1}{t} \int_0^t \delta_{\hat X^i_s(\omega)} ds\right) = \hat m\qquad {\rm in} \; \Pk,
\ee
the time  $\tau$ is finite a.s. .
It follows that
$$
 \P \left[ \inf \left\{t\geq T, \; \sup_{j\in\{1, \dots, N\} }  {\bf d}_1( \pi\sharp\left(\frac{1}{t}\int_0^t \delta_{\hat X^j_s}ds\right), \hat m) \geq \delta\right\}<+\infty \right]=\P\left[ \tau \geq T \right] \to 0 \qquad {\rm as}\; T\to +\infty.
$$
So, given a fixed $K>0$ to be chosen below, we can choose $T$ large enough, depending on $N$, $\delta$, $\ep$  and $K$, such that 
$$
\P\left[ \theta <+\infty\right]\leq  K^{-1}\ep. 
$$
Recalling \eqref{FLgeq0},  we have
\begin{align*}
 J^i((\beta^{N,T,\delta,j})) & \leq  
  \limsup_{t\to+\infty} \frac{1}{t} \E\left[ {\bf 1}_{ \theta =\infty} \int_0^t L(\hat \alpha (\hat X^i_s),\hat X^i_s) + F(\pi\sharp m^{N,i}_{{\bf \hat X}_s})\ ds  \right] \\
& \qquad  + 
\limsup_{t\to+\infty} \frac{1}{t} \E\left[ {\bf 1}_{ \theta < \infty} \int_0^t L( \beta^{N,T,\delta,i} ({\bf X}_\cdot)_s, X^i_s)+ F(\pi\sharp m^{N,i}_{{\bf  X}_s})\ ds \right] \\
&\leq  \limsup_{t\to+\infty} \frac{1}{t} \E\left[ \int_0^t L(\hat \alpha (\hat X^i_s),\hat X^i_s) + F(\pi\sharp m^{N,i}_{{\bf \hat X}_s})\ ds  \right] + C\P[ \theta<+\infty],
\end{align*}
where $C= \max_{x\in \T^d} |L(\hat \alpha(x),x)|+ |L( \alpha^n(x),x)| + \|F\|_\infty$. By \eqref{limhatXi}, we have
\begin{align*}
& \lim_{t\to+\infty} \frac{1}{t}\int_0^t L(\hat \alpha (\hat X^i_s),\hat X^i_s) + F(\pi\sharp m^{N,i}_{{\bf \hat X}_s})\ ds   \\
& \qquad \qquad = 
\inte L(\hat\alpha(x),x) \hat m(x)dx + \int_{(\T^d)^{N-1}} F( m^{N,i}_{{\bf x}}) \Pi_{j\neq i} \hat m(x_j)dx_j =e^N.
\end{align*}
So 
\begin{align*}
 J^i((\beta^{N,T,\delta,j})) & \leq  e^N + C\P[ \theta<+\infty]\leq e^N+ C K^{-1}\ep \leq e^N+\ep,
\end{align*}
if we choose $K=C$. One can show in a similar way that $ J^i((\beta^{N,T,\delta,j}))  \geq  e^N-2\ep$,
which  proves that \eqref{payoffbetaNTdelta} holds. 
\end{proof}

Next we estimate the cost of player $1$ (to fix the ideas) if she deviates and plays some strategy $\beta$ instead of $\beta^{N,T,\delta,1}$.

\begin{Lemma}\label{lem2} Let $\beta$ be a generalized Markov strategy for player 1. Then 
$$
 J^1(\beta, (\beta^{N,T,\delta,j})_{j\neq 1} ) \geq e^N-\ep. 
 $$
\end{Lemma}

\begin{proof}
Let   $(X^j)$ be the solution of the system 
\begin{align}
& d X^j_t= \beta^{N,T,\delta,i} ({\bf X}_t)dt +dB^j_t, \; X^j_0=x^j_0\qquad {\rm if }\; j\geq 2, \notag\\ 
& d X^1_t= \beta({\bf X}_t)dt +dB^1_t, \; X^1_0=x^1_0. \label{X^1}
\end{align}
We set $\theta:= \theta({\bf X})$ where ${\bf X}=(X^1,\dots, X^N)$ and chose again a probability space on which $X^1,\ldots,X^N$ and $\hat X^1,\ldots,\hat X^N$ are strong solutions of \eqref{X^1} and \eqref{hatX} respectively, and therefore  $X^j_{s\wedge\theta}=\hat X^j_{s\wedge\theta}$ a.s., for all $s\geq 0$ and $j\geq 2$.  So  
\begin{align}
 J^1(\beta, (\beta^{N,T,\delta,j})_{j\neq 1} ) & =  \limsup_{t\to+\infty}  \E\Bigl[ {\bf 1}_{\theta=+\infty} \frac{1}{t} \int_0^{t}  L(\beta({\bf X}_\cdot),X^1_s) + F(\pi\sharp m^{N,1}_{{\bf \hat X}_s})\ ds \notag \\
& \qquad  \qquad + 
{\bf 1}_{ \theta < \infty} \frac{1}{t}\int_{0}^t L( \beta({\bf X}_\cdot)_s, X^1_s)+ F(\pi\sharp m^{N,1}_{{\bf  X}_s})\ ds \Big] .\label{ahjlezrsn:fdk}
\end{align}
We evaluate successively all the terms in the right-hand side. For the first term, we claim that, a.s. in $\{\theta=+\infty\}$ and for $t$ large enough, we have
\begin{align}
& \frac{1}{t} \int_0^{t}  L(\beta({\bf X}_\cdot),X^1_s) + F(\pi\sharp m^{N,1}_{{\bf \hat X}_s})\ ds \geq e^N -\ep.
 \label{iaentrdBIS}
\end{align}
For the second term, we are going to prove that a.s. in $\{\theta<+\infty\}$ and for $t$ large enough, we have 
\be\label{iaentrd2}
\frac{1}{t}\int_{0}^t L( \beta({\bf X}_\cdot)_s, X^1_s) ds  \geq -\lambda_0-\ep/2
\ee
and 
\be\label{iaentrd3}
 \frac{1}{t} \int_{t\wedge \theta}^t  F(\pi\sharp m^{N,i}_{{\bf X}_s})\ ds \geq \int_{(\T^d)^{N-1}} F(m^{N,1}_{\bf x})  m^n(dx_2)\dots  m^n(dx_N) -\ep/2. 
\ee

{\bf Proof of \eqref{iaentrdBIS}.} By Lemma \ref{lemma2}, we can choose $\delta>0$ small enough (depending on $\ep$  but not on $T$) such that, for any closed Borel probability measure $\sigma$ on $\R^d\times \T^d$ with second marginal $\mu$
such that ${\bf d}_1(\mu, \hat m)\leq \delta$, we have 
\be\label{izlrnedglfc}
 \int_{\R^d\times \T^d} L( a, x) \sigma (da, dx) \geq \int_{\T^d} L(\hat \alpha(x), x) \hat m(dx) -\ep/3. 
\ee
Let us recall that $\sigma$ being closed means that  $\sigma$ satisfies \eqref{closedcond}. 

 We claim that, a.s. on the event $\{\theta=+\infty\}$ and for  any $t$ large enough, we have
\be\label{iaentrd}
\frac{1}{t} \int_0^{t}  L(\beta({\bf X}_\cdot),X^1_s) ds \geq \int_{\T^d} L(\hat \alpha(x), x) \hat m(dx) -2\ep/3.
\ee
To prove this, we fix  $(\phi_k)$  an enumerable and dense family of $C^2(\T^d)$. Let $\Omega_0$ be the set of $\omega\in \{\theta=+\infty\}$ such that, for any $k\in \N$,  
\be\label{def.Omega0}
\lim_{t\to +\infty}\frac{1}{t} \int_0^{t} D\phi_k(X^1_s)\cdot dB^1_s  =0. 
\ee
By Doob's inequality, $\Omega_0$ has a full probability in $\{\theta=+\infty\}$. Let us now argue by contradiction and assume that, for some $\omega\in \Omega_0$, there exists a sequence $t_n\to+\infty$ such that 
\be\label{iaentrdNO}
\frac{1}{t_n} \int_0^{t_n} L(\beta({\bf X}_\cdot(\omega)),X^1_s(\omega)) ds < \int_{\T^d} L(\hat \alpha(x), x) \hat m(dx) -2\ep/3.
\ee
Let $\sigma_n=\sigma_n(\omega)$ be the Borel probability measure defined on $\R^d\times \T^d$ by 
$$
\int_{\R^d\times \T^d} \phi(a,x) \sigma_n(da,dx)  = \frac{1}{t_n} \int_0^{t_n} \phi(\beta({\bf X}_\cdot(\omega)),X^1_s(\omega)) ds, \qquad \forall \phi \in C^0_b(\R^d\times \T^d).
$$
Let us denote by $\mu_n$ the second marginal of $\sigma_n$. Then, by the definition of $\theta$ and as $\theta=+\infty$, we have ${\bf d}_1(\mu_n, \hat \mu)\leq \delta$ as soon as $t_n\geq T$. Note that, by coercivity of $L$ and  \eqref{iaentrdNO}, the sequence $\sigma_n$ is tight. Hence there exists a Borel probability measure $\sigma$ on $\R^d\times \T^d$ and a subsequence, denoted in the same way, such $\sigma_n$ converges to $\sigma$. Let $\mu$ be the second marginal of $\sigma$. Then ${\bf d}_1(\mu, \hat m)\leq \delta$. Let us check that $\sigma$ is closed.  Indeed, we have, for any $k\in \N$,   
\begin{align*}
& \int_{\R^d\times \T^d} (\Delta \phi_k(x) + a\cdot D\phi_k(x)) \sigma_n(da,dx) = 
  \frac{1}{t_n}  \int_0^{t_n} (\Delta \phi_k(X^1_s) + \beta({\bf X}_\cdot) \cdot D\phi_k(X^1_s)) ds  \\
& \qquad = \frac{1}{t_n}\Bigl[ \phi_k(X^1_{t_n}) -\phi_k(x^1_0) -\sqrt{2} \int_0^{t_n} D\phi_k(X^1_s))\cdot dB^1_s  \Bigr]  
\end{align*}
which tends to $0$ a.s. as $n\to+\infty$ thanks to \eqref{def.Omega0}. So, for any $k\in \N$, 
\begin{align*}
& \int_{\R^d\times \T^d} (\Delta \phi_k(x) + a\cdot D\phi_k(x)) \sigma(da,dx,\omega) = 0.
\end{align*}
By the density of the ($\phi_k$), this proves that the measure $\sigma$ is closed. Letting $n\to+\infty$ in \eqref{iaentrdNO}, we also have, by our convexity assumption on $L$ in \eqref{HypL}, 
$$
\int_{\R^d\times \T^d}  L(a,x) \sigma(da,dx) \leq  \int_{\T^d} L(\hat \alpha(x), x) \hat m(dx) -2\ep/3. 
$$
(See, e.g., Corollary 3.2.3. in \cite{fathi2008weak}). This contradicts \eqref{izlrnedglfc}. So,  for any $\omega\in \Omega_0$ and for $t$ large enough, \eqref{iaentrd} holds.

As the measure on $(\T^d)^{N-1}$ defined by 
$$
(\pi,\dots, \pi)\sharp\left( \frac{1}{t} \int_{0}^t \delta_{(\hat X^2_s,\dots, \hat X^N_s)}ds\right)
 $$
 converges a.s. to the unique invariant measure on $(\T^d)^{N-1}$  associated with the drift $(x_2, \dots, x_N)\to (\hat \alpha(x_2), \dots, \hat \alpha(x_N))$, which is $ \hat\mu(dx_2)\otimes \dots \otimes   \hat \mu(dx_N)$, we have, a.s. and for $t$ large enough, 
\begin{align*}
& \frac{1}{t} \int_0^{t\wedge \theta}  L(\beta({\bf X}_\cdot),X^1_s) + F(\pi\sharp m^{N,1}_{{\bf \hat X}_s})\ ds \\
& \qquad \geq \int_{\T^d} L(\hat \alpha(x), x) \hat \mu(dx) +\int_{(\T^d)^{N-1}} F(m^{N,1}_{\bf x}) \hat \mu(dx_2)\dots \hat \mu(dx_N) -\ep = e^N -\ep.\notag
\end{align*}
This is \eqref{iaentrdBIS}. \\

{\bf Proof of \eqref{iaentrd2}.} We now turn to the estimate of the term $\{\theta<+\infty\}$ in the right-hand side of \eqref{ahjlezrsn:fdk} and first show that \eqref{iaentrd2} holds. For this, we argue as for the proof of \eqref{iaentrd}. We fix  $(\phi_k)$  an enumerable and dense family of $C^2(\T^d)$. Let $\Omega_0$ be the set of $\omega\in \{\theta<+\infty\}$ such that, for any $k\in \N$,  \eqref{def.Omega0} holds. We argue by contradiction, assuming that there exists $\omega \in \Omega_0$ and $t_n\to+\infty$ such that 
\be\label{iaentrdNO2}
\frac{1}{t_n}\int_{0}^{t_n} L( \beta({\bf X}_\cdot)_s, X^1_s)ds  < -\lambda_0-\ep/2.
\ee
Exactly as above, let us define the measure $\sigma_n=\sigma_n(\omega)$ as the Borel probability measure on $\R^d\times \T^d$ such that 
$$
\int_{\R^d\times \T^d} \phi(a,x) \sigma_n(da,dx)  = \frac{1}{t_n} \int_0^{t_n} \phi(\beta({\bf X}_\cdot),X^1_s) ds, \qquad \forall \phi \in C^0_b(\R^d\times \T^d).
$$
By the coercivity of $L$ and assumption \eqref{iaentrdNO2}, the sequence $\sigma_n$ is tight and we can find  $\sigma$ and a subsequence, denoted in the same way, such $\sigma_n$ converges to $\sigma$. We can check as above that $\sigma$ is closed. Letting $n\to+\infty$ in \eqref{iaentrdNO2}, we also have, by convexity of $L$ with respect to the first variable, 
$$
\int_{\R^d\times \T^d}  L(a,x) \sigma(da,dx) \leq -\lambda_0  -\ep/2. 
$$
This contradicts the characterization of $\lambda_0$ in \eqref{eq.characlambda0} and \eqref{iaentrd2} holds in $\Omega_0$.\\

{\bf Proof of \eqref{iaentrd3}.} Next we note that, on $\{\theta<+\infty\}$, we have,  for $j\neq i$,
$$
dX^j_s =  \alpha^n(X^j_s)ds +dB^j_s, \qquad {\rm for}\;  s\geq \theta.
$$
So, the measure on $(\T^d)^{N-1}$ defined by 
$$
(\pi,\dots , \pi) \sharp\left( \frac{1}{t} \int_{0}^t \delta_{(X^2_s,\dots, X^N_s)}ds\right)
 $$
 converges a.s. to the unique invariant measure associated with the drift
 $$
 (x_2, \dots, x_N)\to (\alpha^n(x_2), \dots,  \alpha^n(x_N)),
 $$ 
 which is $ m^n(dx_2)\otimes \dots \otimes  m^n(dx_N)$. So,  for $t$ large enough, \eqref{iaentrd3} holds. \\

{\bf Conclusion.} We now collect our estimates to evaluate the RHS of \eqref{ahjlezrsn:fdk}. As all the terms in the RHS of \eqref{ahjlezrsn:fdk} are bounded below, we have, by Fatou and by \eqref{iaentrdBIS}, \eqref{iaentrd2} and \eqref{iaentrd3}, 
\begin{align*}
 J^1(\beta, (\beta^{N,T,\delta,j})_{j\neq 1} ) & \geq \Bigl(e^N  -\ep\Bigr)\P\left[ \theta=+\infty\right]\\
& \qquad  +\Bigl(-\lambda_0+ \int_{(\T^d)^{N-1}} F(m^{N,1}_{\bf x})   m^n(dx_2)\dots   m^n(dx_N) -\ep\Bigr)\P\left[\theta<+\infty\right].
\end{align*}
By \eqref{cond.ineq}, this proves that $\ds
 J^1(\beta, (\beta^{N,T,\delta,j})_{j\neq 1} )  \geq e^N  -\ep .$
\end{proof}

\begin{proof}[Proof of Theorem \ref{thm:main}] In view of Lemma \ref{lem1} and Lemma \ref{lem2}, the strategies $(\beta^{N,T,\delta,j})$ satisfy the conditions in Definition \ref{def.Nash} with symmetric payoff $(e^N,\cdots, e^N)$. As, by \eqref{eNcve}, $e^N$ converges to $e$, this proves the theorem. 
\end{proof}


\begin{thebibliography}{abc99xyz}
\bibitem{ABC} Arapostathis, A., Biswas, A., \& Carroll, J. (2017). On solutions of mean field games with ergodic cost. Journal de Mathématiques Pures et Appliquées, 107(2), 205-251.

\bibitem{BaCo} Bayraktar, E., \& Cohen, A. (2018). Analysis of a finite state many player game using its master equation. SIAM Journal on Control and Optimization, 56(5), 3538-3568.

\bibitem{bensoussan2013mean}
A.~Bensoussan, J.~Frehse, and P.~Yam, {\sc Mean field games and mean
  field type control theory}, vol.~101, Springer, 2013.
  
\bibitem{BCR} Buckdahn, R., Cardaliaguet, P., and Rainer, C. (2004). Nash equilibrium payoffs for nonzero-sum stochastic differential games. SIAM journal on control and optimization, 43(2), 624-642.  


\bibitem{C} Cardaliaguet, P. (2010). Notes on mean field games. Technical report.

\bibitem{CDLL} Cardaliaguet P., Delarue F.,  Lasry J.-M., Lions P.-L.
 {\sc The master equation and the convergence problem in mean field games.} To appear in Annals of Mathematics Studies. 
 
\bibitem{Ca17}  Cardaliaguet, P. (2017). The convergence problem in mean field games with a local coupling. Applied Mathematics \& Optimization, 76(1), 177-215.
 
\bibitem{CaPo} Cardaliaguet, P., \& Porretta, A. (2017). Long time behavior of the master equation in mean-field game theory. To appear in Analysis and PDEs. arXiv preprint arXiv:1709.04215.

\bibitem{CaRa} Cardaliaguet, P., Rainer, C. (2018). On the (in) efficiency of MFG equilibria. To appear in Sicon.  arXiv preprint arXiv:1802.06637.

\bibitem{CDbook} Carmona, R.,  Delarue, F. (2018). {\sc Probabilistic Theory of Mean Field Games with Applications I-II}. Springer Nature.


\bibitem{CePe} Cecchin, A., \& Pelino, G. (2018). Convergence, fluctuations and large deviations for finite state mean field games via the master equation. To appear in Stochastic Processes and their Applications.

\bibitem{CDPFP} Cecchin, A., Dai Pra, P., Fischer, M., \& Pelino, G. (2018). On the convergence problem in Mean Field Games: a two state model without uniqueness. arXiv preprint arXiv:1810.05492.

\bibitem{DLR} Delarue, F., Lacker, D., \& Ramanan, K. (2018). From the master equation to mean field game limit theory: A central limit theorem. arXiv preprint arXiv:1804.08542.

\bibitem{DLR2} Delarue, F., Lacker, D., \& Ramanan, K. (2018). From the master equation to mean field game limit theory: Large deviations and concentration of measure. arXiv preprint arXiv:1804.08550.

\bibitem{DeTc} Delarue, F., \& Tchuendom, R. F. (2018). Selection of equilibria in a linear quadratic mean-field game. arXiv preprint arXiv:1808.09137.

\bibitem{fathi2008weak} Fathi A.  (2008)
{\sc Weak kam theorem in lagrangian dynamics.} Preliminary version
  number 10, by CUP.


\bibitem{Fe} Feleqi, E. (2013). The derivation of ergodic mean field game equations for several populations of players. Dynamic Games and Applications, 3(4), 523-536.


\bibitem{Fi} Fischer, M. (2017). On the connection between symmetric $ n $-player games and mean field games. The Annals of Applied Probability, 27(2), 757-810.

\bibitem{GV} Gomes, D. A., \& Valdinoci, E. (2011). Duality Theory, Representation Formulas and Uniqueness Results for Viscosity Solutions of Hamilton-Jacobi Equations. In Dynamics, Games and Science II (pp. 361-386). Springer, Berlin, Heidelberg.

\bibitem{HMC} Huang, M., Malhamé, R. P., \& Caines, P. E. (2006). Large population stochastic dynamic games: closed-loop McKean-Vlasov systems and the Nash certainty equivalence principle. Communications in Information \& Systems, 6(3), 221-252.

\bibitem{Ko76}  Kononenko, A. F. (1976). On equilibrium positional strategies in nonantagonistic differential games. In Doklady Akademii Nauk (Vol. 231, No. 2, pp. 285-288). Russian Academy of Sciences.


\bibitem{La16} Lacker, D. (2016). A general characterization of the mean field limit for stochastic differential games. Probability Theory and Related Fields, 165(3-4), 581-648.

\bibitem{La18} Lacker, D. (2018). On the convergence of closed-loop Nash equilibria to the mean field game limit. arXiv preprint arXiv:1808.02745.

%

\bibitem{LL07mf}  J.-M. Lasry and P.-L. Lions, Mean field games, 
  Japanese Journal of Mathematics, 2 (2007), pp.~229--260.

  
\bibitem{Nu} Nutz, M. (2018). A mean field game of optimal stopping. SIAM Journal on Control and Optimization, 56(2), 1206-1221.

\end{thebibliography}
\end{document}